\documentclass[12pt,a4paper]{article}

\usepackage{algorithm}
\usepackage{booktabs}
\usepackage{graphicx}
\usepackage{amsmath}
\usepackage{amssymb}
\usepackage{latexsym}
\usepackage{crop}
\usepackage{algorithmic,algorithm}
\usepackage{multirow}
\usepackage{bm}
\usepackage{bbm}
\usepackage{enumerate}
\usepackage{url}
\usepackage{array}
\usepackage{paralist}
\usepackage{diagbox}
\usepackage{wasysym}
\usepackage{booktabs}
\usepackage[dvipsnames]{xcolor}
\usepackage[colorlinks = true, pdfstartview = FitV, linkcolor = blue, citecolor = blue, urlcolor = blue]{hyperref}

\usepackage{listings}

\usepackage{rotating}

\usepackage[capitalise]{cleveref}
\crefname{equation}{}{}
\crefname{figure}{Figure}{Figures}
\creflabelformat{equation}{\textup{(#2#1#3)}}
\crefname{assumption}{Assumption}{Assumptions}
\crefname{condition}{Condition}{Conditions}

\usepackage{xspace}

\usepackage{fullpage}
\usepackage{multirow}

\usepackage[sort,nocompress]{cite}

\usepackage{arydshln}
\usepackage{enumitem}
\setlist[enumerate,1]{leftmargin=*,wide=0em, noitemsep,nolistsep, label = {\bfseries \arabic*.}}
\setlist[itemize,1]{leftmargin=*,wide=0em, noitemsep,nolistsep}

\usepackage{pifont}
\makeatletter
\newcommand*{\transpose}{%
	{\mathpalette\@transpose{}}%
}
\newcommand*{\@transpose}[2]{%
	\raisebox{\depth}{$\m@th#1\intercal$}%
}
\makeatother

\newcommand*\xbar[1]{%
	\hbox{%
		\vbox{%
			\hrule height 0.5pt 
			\kern0.5ex
			\hbox{%
				\kern-0.1em
				\ensuremath{#1}%
				\kern-0.1em
			}%
		}%
	}%
} 

\definecolor{forestgreen}{rgb}{0.13, 0.55, 0.13}

\definecolor{amber}{rgb}{1.0, 0.75, 0.0}

\definecolor{bananayellow}{rgb}{.8, 0.6, 0}

\newcounter{comment}

\usepackage{amsthm}
\usepackage[framemethod=TikZ]{mdframed}

\mdfdefinestyle{theoremstyle}{%
	linewidth = 1pt,%
	roundcorner = 10pt,%
	leftmargin = 0,%
	rightmargin = 0,%
	backgroundcolor = cyan!3,%
	outerlinecolor = magenta!70!black,%
	splittopskip = \topskip,%
	ntheorem = true,%
}
\mdtheorem[style=theoremstyle]{claim}{Claim}

\newmdtheoremenv[%
linewidth = 1pt,%
roundcorner = 10pt,%
leftmargin = 0,%
rightmargin = 0,%
backgroundcolor = green!3,%
outerlinecolor = blue!70!black,%
splittopskip = \topskip,%
ntheorem = true,%
]{theorem}{Theorem}

\newmdtheoremenv[%
linewidth = 1pt,%
roundcorner = 10pt,%
leftmargin = 0,%
rightmargin = 0,%
backgroundcolor = green!3,%
outerlinecolor = blue!70!black,%
splittopskip = \topskip,%
ntheorem = true,%
]{corollary}{Corollary}

\newmdtheoremenv[%
linewidth = 1pt,%
roundcorner = 10pt,%
leftmargin = 0,%
rightmargin = 0,%
backgroundcolor = green!3,%
outerlinecolor = blue!70!black,%
splittopskip = \topskip,%
ntheorem = true,%
]{lemma}{Lemma}

\newmdtheoremenv[%
linewidth = 1pt,%
roundcorner = 10pt,%
leftmargin = 0,%
rightmargin = 0,%
backgroundcolor = yellow!3,%
outerlinecolor = blue!70!black,%
splittopskip = \topskip,%
ntheorem = true,%
]{definition}{Definition}

\newmdtheoremenv[%
linewidth = 1pt,%
roundcorner = 10pt,%
leftmargin = 0,%
rightmargin = 0,%
backgroundcolor = green!3,%
outerlinecolor = blue!70!black,%
splittopskip = \topskip,%
ntheorem = true,%
]{proposition}{Proposition}

\newmdtheoremenv[%
linewidth = 1pt,%
roundcorner = 10pt,%
leftmargin = 0,%
rightmargin = 0,%
backgroundcolor = green!3,%
outerlinecolor = blue!70!black,%
splittopskip = \topskip,%
ntheorem = true,%
]{condition}{Condition}

\newmdtheoremenv[%
linewidth = 1pt,%
roundcorner = 10pt,%
leftmargin = 0,%
rightmargin = 0,%
backgroundcolor = green!3,%
outerlinecolor = blue!70!black,%
splittopskip = \topskip,%
ntheorem = true,%
]{assumption}{Assumption}

\theoremstyle{definition}
\newmdtheoremenv[%
linewidth = 1pt,%
roundcorner = 10pt,%
leftmargin = 0,%
rightmargin = 0,%
backgroundcolor = blue!3,%
outerlinecolor = blue!70!black,%
splittopskip = \topskip,%
ntheorem = true,%
]{example}{Example}

\theoremstyle{definition}
\newmdtheoremenv[%
linewidth = 1pt,%
roundcorner = 10pt,%
leftmargin = 0,%
rightmargin = 0,%
backgroundcolor = red!3,%
outerlinecolor = blue!70!black,%
splittopskip = \topskip,%
ntheorem = true,%
]{remark}{Remark}

\usepackage{tikz}
\usepackage{xparse}

\NewDocumentCommand\DownArrow{O{2.0ex} O{black}}{%
	\mathrel{\tikz[baseline] \draw [<-, line width=0.5pt, #2] (0,0) -- ++(0,#1);}
}

\usepackage{listings} 

\definecolor{mygreen}{rgb}{0,0.6,0}
\definecolor{mygray}{rgb}{0.5,0.5,0.5}
\definecolor{mymauve}{rgb}{0.58,0,0.82}

\usepackage{dsfont}

\usepackage[latin1]{inputenc}
\usepackage{amsmath}
\usepackage{amsfonts}
\usepackage{amssymb}
\usepackage{makeidx}
\usepackage{graphicx}
\usepackage{caption}
\usepackage{subcaption}
\usepackage{framed}
\usepackage{booktabs,array}
\usepackage{xcolor}

\begin{document}
	

\title{Surprise Maximization: \\ A Dynamic Programming Approach}
\author{
	Ali Eshragh\thanks{School of Mathematical and Physical Sciences, University of Newcastle, NSW, Australia, and International Computer Science Institute, Berkeley, CA, USA. Email:  \tt{ali.eshragh@newcastle.edu.au}}
}
\date{}
\maketitle


\begin{abstract}
	Borwein et al. \cite{Borwein2000} solved a ``surprise maximization'' problem by applying results from convex analysis and mathematical programming. Although, their proof is elegant, it requires advanced knowledge from both areas to understand it. Here, we provide another approach to derive an optimal solution of the problem by utilizing dynamic programming. 
\end{abstract}

\section{Introduction}

Borwein et al. \cite{Borwein2000} introduced an optimization problem on maximizing the expected value of the surprise function. More precisely, they exploited results from convex analysis and mathematical programming to find an optimal solution of the following non-linear programming model, called \textsf{SM1}:
\begin{align*}
	\medskip \mbox{maximize}\ \ S_m(p_1,\ldots,p_m) & = \displaystyle{\sum_{j=1}^m p_j \log \frac{p_j}{\frac{1}{m} \sum_{i=j}^m p_i}-\sum_{j=1}^m p_j}\\
	\medskip \mbox{subject to}\ \ \ \ \ \ \ \ \ \ \ \ \sum_{j=1}^m p_j & = 1\,, \\
	\medskip  \mbox{and}\ \ \ \ \ \ \ \ \ \ \ \ \ \ \ \ \ \ \ \ \ \ \ \ \ p_j & \geq 0\ \ \ \mbox{for }j=1,\ldots,m\,.
 \end{align*}

Here, a \emph{dynamic programming} approach is utilized to find an optimal solution of the \textsf{SM1} model. First of all, we simplify the objective function $S_m(p_1,\ldots,p_m)$ as follows: 
\begin{align*}
	\medskip S_m(p_1,\ldots,p_m) & = \displaystyle{\sum_{j=1}^m p_j \left(\log p_j-\log\left(\sum_{i=j}^m p_i\right)\right) +\log m -1}
\end{align*}

\noindent Without loss of generality, we can disregard the constant term $\log m -1$ and carry out our optimisation over terms involving the variables $p_j$. Thus, we focus on the following optimisation model, called \textsf{SM2}: 
\begin{align*}
	\medskip \mbox{maximize}\ \ \widetilde{S_m}(p_1,\ldots,p_m) & := \displaystyle{\sum_{j=1}^m p_j \left(\log p_j-\log\left(\sum_{i=j}^m p_i\right)\right)}\\
	\medskip \mbox{subject to}\ \ \ \ \ \ \ \ \ \ \ \ \sum_{j=1}^m p_j & = 1\,, \\
	\medskip  \mbox{and}\ \ \ \ \ \ \ \ \ \ \ \ \ \ \ \ \ \ \ \ \ \ \ \ \ p_j & \geq 0\ \ \ \mbox{for }j=1,\ldots,m\,.
 \end{align*}

Now consider the following counterpart investment problem: Suppose that we are given $1$ unit of money to invest in $m$ consecutive days. If we spend $p_1,\ldots,p_m$ units of money in days $1,\ldots,m$, then the total return of this investment will be given by $\widetilde{S_m}(p_1,\ldots,p_m)$. We want to find an optimal investment policy such that the total return over $m$ days is maximised. Clearly, \textsf{SM2} solves this optimal investment problem (This problems is called \emph{optimal resource allocation problem} in the literature).

\section{Dynamic Programming}

We apply a \emph{dynamic programming} approach to solve \textsf{SM2}. In this model, the \emph{stage} is each investment opportunity (i.e., day) and the state of the system is the remained amount of money to invest in the subsequent stages. Let $V_j(r)$ denote the maximum total return over days $j,\ldots,m$ while $r$ units of money remained (i.e., $1-r$ units have been already spent in days $1,\ldots,j-1$). The \emph{Bellman optimality equation} is given as follows: 
\begin{align}\label{EquBOE}
	\begin{cases}
		\medskip \displaystyle{V_j(r)\ =\ \max_{0\leq x \leq r} \{x\log x -x\log r + V_{j+1}(r-x)\}\ \ \ \mbox{for}\ j=1,\ldots,m-1}, \\
		\medskip V_m(r)\ =\ 0.
	\end{cases}
\end{align}
Let $p_j^*(r)$ denote an optimal investment policy in day $j$ when the stage of the system is $r$. Obviously, we have $p_m^*(r)=r$. So, the optimity equation \cref{EquBOE} for $j=m-1$ is solved as follows. 
\begin{align*}
	\medskip V_{m-1}(r) & = \displaystyle{\max_{0\leq x \leq r} \{x\log x -x\log r + V_{m}(r-x)\}} \\
	\medskip & = \displaystyle{\max_{0\leq x \leq r} \{x\log x -x\log r \}}\,.
\end{align*}
Since the function $h(x):=x\log x -x\log r$ is a convex function over interval $[0,r]$, its maximum coincides with its extremum. Thus, 
\begin{align}
	\medskip \label{sol:industion_hyp1} p_{m-1}^*(r) & = r e^{-1} \\
	\medskip \label{sol:industion_hyp2} V_{m-1}(r) & = -p_{m-1}^*(r)\ =\ -r e^{-1}.
\end{align}

Solving the optimality equation \eqref{EquBOE} for $j=m-2$ gives a trend in the optimal investment policy, summarized in following theorem.
\begin{theorem}
For the optimality equation \eqref{EquBOE},
\begin{align}\label{sol:opt_pol}
	\medskip p_{j}^*(r) & = r e^{-\gamma_j},
\end{align}
where 
\begin{align}\label{EquGam}	
	\begin{cases}
		\medskip \displaystyle{\gamma_{j-1}\ =\ \gamma_j+e^{-\gamma_j}\ \ \ \mbox{for}\ j=1,\ldots,m-1}, \\
		\medskip \gamma_m\ =\ 0,
	\end{cases}
\end{align}
is an optimal investment policy. Moreover, the optimal value is given by
\begin{align}\label{sol:opt_val}
	\medskip V_{j}(r) & = -\sum_{i=j}^{m-1} p_i^*(r).
\end{align}
\end{theorem}

\begin{proof}
We prove this theorem by induction. It is readily seen that \eqref{sol:industion_hyp1} and \eqref{sol:industion_hyp2} satisfies \eqref{sol:opt_pol} and \eqref{sol:opt_val} for $j=m-1$, respectively. Now, assume that the latter optimal policy and optimal value are correct for $j\geq k$ and we show that they hold for $j=k-1$, as well. By considering the induction assumption, we have
\begin{align*}
	\medskip V_{k-1}(r) & = \displaystyle{\max_{0\leq x \leq r} \{x\log x -x\log r + V_{k}(r-x)\}} \\
	\medskip  & = \displaystyle{\max_{0\leq x \leq r} \{x\log x -x\log r -\sum_{i=k}^{m-1} p_i^*(r-x)\}} \\
	\medskip  & = \displaystyle{\max_{0\leq x \leq r} \{x\log x -x\log r -\sum_{i=k}^{m-1} (r-x) e^{-\gamma_i}\}}  \\
	\medskip  & = \displaystyle{\max_{0\leq x \leq r} \{x\log x -x\log r -(r-x) (\gamma_{k-1}-1)\}},
\end{align*}
where the last equality is derived by summing up both sides of \eqref{EquGam} over $j=k,\ldots,m-1$. One can see that the latter univariate optimization problem achieves its maximum at $x^* = r e^{-\gamma_{k-1}}$, and the corresponding optimal value $V_{k-1}(r)$ equals to $-\sum_{i=k-1}^{m-1} p_i^*(r)$\,. This completes the proof.
\end{proof}

\begin{corollary}
	An optimal solution of the model \textsf{SM2} is given by:
	\begin{align*}
		p_j^* & = \begin{cases}
			\medskip p_1^*(1) & \mbox{for}\ j=1, \\
			\medskip p_j^*(1-\sum_{i=1}^{j-1}p_i^*) & \mbox{for}\ j=2,\ldots,m.
		\end{cases}
	\end{align*}	
\end{corollary}


\end{document}